\documentclass[a4paper,11pt]{article}

\usepackage[utf8]{inputenc}
\usepackage{amsmath, amssymb, amsthm}
\usepackage{graphicx}
\usepackage{accents}
\usepackage{color}
\usepackage{geometry}
\usepackage{fancyhdr} 
\usepackage{titling}
\usepackage{enumerate}
\usepackage{titlesec}
\usepackage{authblk}

\titlespacing*{\section}
{0pt}{5.5ex plus 1ex minus .2ex}{4.3ex plus .2ex}
\titlespacing*{\subsection}
{0pt}{5.5ex plus 1ex minus .2ex}{4.3ex plus .2ex}

\newcommand{\subtitle}[1]{%
  \posttitle{%
    \par\end{center}
    \begin{center}\large#1\end{center}
    \vskip0.5em}%
}

\newcommand{\wt}{\widetilde}

\newcommand{\R}{\mathbb R}
 
\newcommand{\N}{\mathbb{N}} 
\newcommand{\Z}{\mathbb{Z}}

\newcommand{\M}{\mathcal{M}}

\newcommand{\Leg}{\mathcal{L}}
\newcommand{\Pone}{\mathcal{P}^1}

\newtheorem{cor}{Corollary}[section]
\newtheorem{defn}[cor]{Definition}
\newtheorem{prop}[cor]{Proposition}
\newtheorem*{prop*}{Proposition}
\newtheorem{thm}[cor]{Theorem}
\newtheorem{exmp}[cor]{Example}
\newtheorem{lem}[cor]{Lemma}
\newtheorem{rmk}[cor]{Remark}

\title{A connecting theorem for geodesic flows on the 2-torus}
\author{Stefan Klempnauer\thanks{stefan.klempnauer@rub.de}}
\affil{Faculty of Mathematics, Ruhr-University Bochum}
\begin{document}
\maketitle

\begin{abstract}
\noindent We use a result of J. Mather on the existence of connecting orbits for compositions of monotone 
twist maps of the cylinder to prove the existence of connecting geodesics on the unit tangent bundle
$ST^2$ of the 2-torus in regions without invariant tori. 
\end{abstract}

The author thanks the SFB CRC/TRR 191 \emph{Symplectic Structures in Geometry, Algebra and Dynamics} of the DFG and the Ruhr-University Bochum for the funding of his research.

\section{Introduction and main result}

Let $T^2 \cong \R^2 / \Z^2$ denote the 2-torus with universal covering $\pi: \R^2 \to T^2$. The tangent bundle
is given by $TT^2 \cong T^2 \times \R^2$. In the following we will recall the definition of a Finsler metric. For a general 
overview of Finsler geometry see \cite{bao00}. 

\begin{defn}
A \emph{Finsler metric} on $T^2$ is a map 
$$
F : TT^2 \to [0,\infty)
$$
with the following properties
\begin{enumerate}
\item (Regularity) $F$ is $C^\infty$ on $TT^2 - 0$ 
\item (Positive homogeneity) $F(x,\lambda y) = \lambda \cdot F(x,y)$ for $\lambda > 0$
\item (Strong convexity) The hessian 
$$
(g_{ij}(x,v)) = \left(\partial_{v_iv_j} \frac{1}{2} F(x,v)^2 \right)
$$
is positive-definite for every $(x,v) \in TT^2 - 0$. 
\end{enumerate}
\end{defn}

A Finsler metric is called \emph{reversible} if $F(x,v) = F(x,-v)$ for every $(x,v) \in TT^2$. The 
\emph{unit tangent bundle} $ST^2 \cong T^2 \times S^1$ is given by $ST^2 = F^{-1}(\{1\})$.\\

We define the \emph{length} $l_F(c)$ of a piecewise differentiable curve $c:[a,b] \to T^2$ via 
$$
l_F(c) = \int_a^b F(\dot c(t)) dt
$$
Lifting the Finsler metric $F$ to $\R^2$ allows us to define the length of a curve $c:[a,b] \to \R^2$ 
in the universal covering in the same manner. 
The \emph{geodesic flow} $\phi^t: ST^2 \to ST^2$ is the restriction of the Euler-Lagrange flow of the 
Lagrangian $L_F$, with 
$$
L_F = \frac{1}{2}F^2
$$
to the unit tangent bundle. A \emph{geodesic} we call either a trajectory of the Euler-Lagrange flow, or 
its projection to $T^2$, i.e. a geodesic is a curve $t \mapsto c(t) \subset T^2$ satisfying the Euler-Lagrange equation
$$
\partial_x L_F(c, \dot c) - \partial_t(\partial_v L_F(c, \dot c)) = 0
$$ 
In general we will assume a geodesic $c$ to be parametrized by arclength, i.e. $\dot c \subset ST^2$. 
Sometimes it can be convenient to consider lifts $\wt c : \R \to \R^2$ or $ \dot{\wt c}: \R \to T\R^2$
of a geodesic to the universal cover. These can be seen as
the geodesics of the lifted $\Z^2$-periodic Finsler metric $\wt F$ on $\R^2$.

\begin{defn}
A subset $\Lambda \subset ST^2$ is called an \emph{invariant torus} if $\Lambda$ is the graph
of a continuous map $X: T^2 \to S^1$ and $\phi^t$-invariant.
An invariant torus (or more generally any $\phi^t$-invariant set) $\Lambda$ has \emph{bounded direction}
(with respect to $v \in \Z^2 -\{0\}$), if the lifts $ \tilde c: \R \to \R^2$ of geodesics in $\Lambda$ are graphs over
 the euclidean line $\R v$. 
\end{defn}

In analogy to a Birkhoff region of instability we define the following notion of an instability region
of the geodesic flow. 

\begin{defn}
An \emph{instability region} is a compact invariant subset $U \subset ST^2$ with boundary being
the disjoint union of two invariant tori $\Lambda_-, \Lambda_+$, such that every invariant torus 
$\Lambda \subset U$ is equal to $\Lambda_-$ or $\Lambda_+$. 
\end{defn}

The main result is the following. 

\begin{thm} \label{main}
Let $F$ be a Finsler metric on $T^2$ and let $U \subset ST^2$ be an instability region with bounded direction 
(with respect to $e_1$). Furthermore, we assume that there are no closed geodesics
in the boundary of $U$.
Then there exists an $F$-geodesic $c: \R \to T^2$ that connects the two boundary components of $U$ (i.e. there exist sequences $t_n \to \infty$ and $s_n \to -\infty$ with $\lim_{n \to \infty} \dot c(t_n) \in \Lambda_+$ and
$\lim_{n \to \infty} \dot c(s_n) \in \Lambda_-$, where $\Lambda_-$ and $\Lambda_+$ are the two boundary components of $U$).
\end{thm}

\section{Return maps of the geodesic flow on $T^2$ and twist maps}

\subsection{Return maps as Lagrangian time-1 maps}

Let $F$ be a Finsler metric on the 2-torus $T^2$ with $\wt F$ denoting the lifted metric on the universal cover 
$\R^2$. For a vector $w = (w_1,w_2) \in \R^2$ we will write $w^\perp$ for the vector 
$w^\perp = (-w_2, w_1)$.

\begin{defn} \label{lagrangiandefn}
For a Finsler metric $F$ and a prime element $v \in \Z^2 \backslash \{0\}$ we define a time-dependent Lagrangian 
on $\R$ via 
$$
\wt L(t,x,r) = \wt F(tv + xv^\perp, v + rv^\perp)
$$ 
which is 1-periodic in the time $t$ and 
1-periodic in the first component $x$. The periodicity is a direct consequence of the $\Z^2$-periodicity
of $\wt F$. Since $\wt L$ is periodic in $x$ it can be interpreted as a time-periodic Lagrangian $L$ on $S^1$.
Note that $\wt L$ is strictly convex, i.e. $\partial_{rr} \wt L > 0$ because of the strict convexity of the 
Finsler metric $\wt F$.  
\end{defn}

We have a correspondence of lifted geodesics, which are graphs over $v\R$ and solutions of $L$ in the
following theorem by J.P. Schr\"oder. A special case can be found in \cite{schroeder13}.

\begin{thm} \label{schrthm}
Let $F$ and $L$ as above and let $\theta: \R \to \R$ be a smooth function.
Let $\gamma: \R \to \R^2$ be the curve given by
$$
\gamma(t) = tv + \theta(t)v^\perp
$$
Then $\gamma$ is a reparametrization of an $\wt F$-geodesic if and only if $\theta$ is an
Euler-Lagrange solution of $\wt L$. 
\end{thm}

\begin{proof}
Observe that we have the following relation between the Lagrangian action $A_{\wt L}$ and the Finsler length $l_{\wt F}$.
\begin{flalign*}
A_{\wt L}(\theta|_{[a,b]}) &= \int_a^b \wt L(t, \theta(t), \theta'(t)) dt \\
&= \int_a^b \wt F(tv + \theta(t)v^\perp, v + \theta'(t)v^\perp) dt \\
&= \int_a^b \wt F(\gamma(t), \dot \gamma(t)) dt \\
&= l_{\wt F}(\gamma|_{[a,b]})
\end{flalign*}

Assume now that $\gamma:[a,b] \to \R^2$ is the reparametrization of an $\wt F$-geodesic, 
i.e. $\partial_{s=0} l_{\wt F}(\gamma_s) =0$ for any
proper variation of $\gamma$. Let $\theta_s : [a,b] \to \R$  be a proper variation of $\theta$. 
From $\theta_s$ we 
construct a proper variation $\gamma_s$ of $\gamma$ via 
$$
\gamma_s(t) = tv + \theta_s(t) v^\perp
$$ 
Then we have $A_{\wt L}(\theta_s) = l_{\wt F}(\gamma_s)$ for every $s$, and hence we have 
$$
\partial_s|_{s=0} A_{\wt L}(\theta_s) = \partial_s|_{s=0}l_{\wt F}(\gamma_s) = 0.
$$
This proves one direction.
To prove the other direction assume now that $\theta: [a,b] \to \R$ is critical with respect to the Lagrangian action and 
let $X: [a,b] \to \R^2$ be a vector field along $\gamma$ with $X(a) = X(b) = 0$. Since 
$\dot \gamma(t) = v + \theta'(t) v^\perp$ the pair of vectors $\{\dot \gamma(t) , v^\perp \}$ always forms a basis of $\R^2$. 
Thus, we can rewrite the vector field $X$ as 
$$
X(t) = \underbrace{\lambda(t) \dot \gamma(t)}_{A(t)} + \underbrace{\mu(t) v^\perp}_{B(t)}
$$
for functions $\lambda, \mu$ with $\lambda(a) = \lambda(b) = \mu(a) = \mu(b) = 0$.
Let $\gamma_s$ be a proper variation of $\gamma$ corresponding to the variational vector field $X$ and let $\beta_s$
be a proper variation of $\gamma$ corresponding to $B$, i.e.
$$
\partial_s|_{s=0} \gamma_s(t) = X(t) \quad \text{and} \quad \partial_s|_{s=0} \beta_s(t) = B(t)
$$
Observe that for small $|s|$ the curve $\eta_s: t \mapsto \gamma(t + s\lambda(t))$ is a reparametrization of $\gamma$ 
and hence has length independent of $s$. Thus we have

\begin{flalign*}
0 &= \partial_s|_{s=0} l_{\wt F}(\eta_s) \\
&=  \partial_s|_{s=0} \int_a^b \wt F(\eta_s(t), \partial_t \eta_s(t)) dt \\ 
&=  \int_a^b \partial_s|_{s=0} \wt F(\eta_s(t), \partial_t \eta_s(t)) dt \\ 
&= \int_a^b  \partial_1 \wt F(\eta_0(t), \partial_t \eta_0(t)) \partial_s|_{s=0} \eta_s(t) 
     + \partial_2 \wt F(\eta_0(t), \partial_t \eta_0(t)) \partial_s|_{s=0} \partial_t \eta_s(t) dt \\ 
&= \int_a^b  \partial_1 \wt F(\gamma(t), \partial_t \gamma(t)) \partial_s|_{s=0} \eta_s(t) 
     + \partial_2 \wt F(\gamma(t), \partial_t \gamma(t)) \partial_t  \partial_s|_{s=0} \eta_s(t) dt \\ 
&= \int_a^b  \partial_1 \wt F(\gamma(t), \partial_t \gamma(t)) A(t)
     + \partial_2 \wt F(\gamma(t), \partial_t \gamma(t)) \dot A(t) dt 
\end{flalign*}

Hence it follows that 

\begin{flalign*}
\partial_s|_{s=0} l_{\wt F}(\gamma_s) &= \partial_s|_{s=0} \int_a^b \wt F(\gamma_s(t), \dot \gamma_s(t)) dt \\
&= \int_a^b \partial_1 \wt F(\gamma_0(t), \dot \gamma_0(t)) \partial_s|_{s=0} \gamma_s(t) 
+ \partial_2 \wt F(\gamma_0(t), \dot \gamma_0(t)) \partial_s|_{s=0} \partial_t \gamma_s(t) \\
&= \int_a^b \partial_1 \wt F(\gamma(t), \dot \gamma(t)) \partial_s|_{s=0} \gamma_s(t) 
+ \partial_2 \wt F(\gamma(t), \dot \gamma(t)) \partial_t \partial_s|_{s=0} \gamma_s(t) \\
&= \int_a^b \partial_1 \wt F(\gamma(t), \dot \gamma(t)) X(t) 
+ \partial_2 \wt F(\gamma(t), \dot \gamma(t)) \dot X(t) \\
&= \int_a^b \partial_1 \wt F(\gamma(t), \dot \gamma(t)) (A(t) + B(t))
+ \partial_2 \wt F(\gamma(t), \dot \gamma(t)) (\dot A(t) + \dot B(t)) \\
&= \int_a^b \partial_1 \wt F(\gamma(t), \dot \gamma(t)) B(t)
+ \partial_2 \wt F(\gamma(t), \dot \gamma(t)) \dot B(t) \\
&= \partial_s|_{s=0} l_{\wt F}(\beta_s)
\end{flalign*}

Consequently, the curve $\gamma$ is critical with respect to the Finsler length if
$$
\partial_s |_{s=0} l_{\wt F}(\beta_s) = 0
$$
for every proper variation $\beta_s$, which varies $\gamma$ only in $v^\perp$ direction, i.e. $\beta_s$ is of the form
$$
\beta_s(t) = tv + \theta_s(t) v^\perp
$$
For those variations we have already computed that if $\theta_s$ is critical with respect to the $\wt L$-action then 
$\beta_s$ is critical with respect to the Finsler length.

\end{proof}

Remember that $T\R^2 \cong \R^2 \times \R^2$.
To establish a connection between return maps of the geodesic flow and the time-1 map
of the Lagrangian $\wt L$ we define the following sets

$$
\wt V = \{(x,w) \in T\R^2 : \langle w,v \rangle > 0, \wt F(w) = 1\}
$$

with sections $\wt V_i$, where $i \in \Z$, given by

$$
\wt V_i = \{(iv + bv^\perp, w) \in T\R^2 : \langle w, v \rangle > 0, \wt F(w) = 1, b \in \R\}
$$ 

In general the Euler-Lagrange solutions of $\wt L$ will not exist for all times. However, because
of theorem \ref{schrthm} we have the following

\begin{prop}
If the Euler-Lagrange flow of $\wt L$ is complete then the set $\wt V$ is invariant and every orbit
starting in a section $\wt V_i$ will consecutively pass through the sections $\wt V_{i+j}$ for $j \in \N$. 
\end{prop}

\begin{proof}
Completeness of the Euler-Lagrange flow means that the Euler-Lagrange solutions of $\wt L$ exist
for all times. Let $(x,w) \in \wt V$.
We write $(x,w) = (av + bv^\perp, rv + hv^\perp)$ for coefficients $a,b,h,r \in \R$ with $r > 0$. Let 
$\wt c: \R \to \R^2$ be the geodesic with $\wt c(a) = x$ and $\dot{ \wt c}(a) = \tfrac{1}{r}w$. To show
that $\wt V$ is invariant we have to show that $\langle \dot{\wt c}(t) , v \rangle > 0$ for every 
$t \in \R$. 
Let $\theta: \R \to \R$ be the Euler-Lagrange solution of $\wt L$ with 
$\theta(a) = b$ and $\theta'(a) = \tfrac{h}{r}$. Since $\theta$ exists for all times it follows from 
Theorem \ref{schrthm} that the curve $\gamma: \R \to \R^2$ with 
$\gamma(t) = tv + \theta(t)v^\perp$ is a reparametrization of a geodesic. 
We have $\gamma(a) = av + bv^\perp = x$ and $\dot \gamma(a) = v + \tfrac{h}{r}v^\perp = 
\tfrac{1}{r}w$ and hence $\gamma$ is a reparametrization of the geodesic $\wt c$. 
Since $\langle \dot \gamma(t) , v \rangle = \Vert v \Vert^2$ for every $t$, we have 
$\langle \dot{\wt c }(t), v \rangle > 0$ for every $t$.
With the same construction one sees that if $(x,w)$ was chosen in $\wt V_i$ the constructed 
reparametrization $\gamma$ will pass through $\wt V_{i+j}$ at times $i + j$.
\end{proof}

By $\wt R_0 : \wt V_0 \to \wt V_1$ we denote the map that maps an element $(x,w) \in \wt V_0$
to the point of intersection of $\wt V_1$ with the orbit passing through $(x,w)$. We define the projected
section $V \subset ST^2$ via 
$$
V := \Pi(V_0)
$$ 
Since the geodesic
flow on $T\R^2$ is $\Z^2$-invariant, the map $\wt R_0$ descends to a map 
$R: V \to V$, which is the $n$-th return map of the section $V$, with respect to the geodesic flow on $ST^2$. Here, $n-1$ is equal
to $\Vert v \Vert^2 - 1$ ($n-1$ is equal to the number of deck-transformations of $\R v^\perp$, that lie
in between $\R v^\perp$ and $v + \R v^\perp$, which corresponds to the number of integer points
in the rectangle spanned by $v$ and $v^\perp$. If $v$ is a prime element in $\Z^2$ then this is 
equal to $\Vert v \Vert^2 - 1$ after Pick's theorem).

\begin{prop} \label{conjugateprop}
If the Euler-Lagrange flow of $\wt L$ is complete
the return map $R$ is conjugated to the time-1 map $\varphi_L^{0,1}: Z \to Z$ of the 
Lagrangian $L$.
\end{prop}

\begin{proof}
Let $(x,w) \in \wt V_0$. We can rewrite $(x,w)$ as 
$$
(x,w) = \left(\frac{\langle x, v^\perp \rangle }{\Vert v \Vert^2} v^\perp , 
\frac{\langle w, v \rangle}{\Vert v \Vert^2} v + \frac{\langle w, v^\perp \rangle}{\Vert v \Vert^2}
v^\perp \right)
$$
Define real numbers $b,h \in \R$ and $r > 0$, such that
$(x,w) = (bv^\perp , rv + hv^\perp)$.
Since the Euler-Lagrange flow is complete there
exists an Euler-Lagrange solution $\theta: \R \to \R$ of $\wt L$ with 
$$
\theta(0) = b \quad \text{and} \quad \theta'(0) = \frac{h}{r}
$$
From theorem \ref{schrthm} it follows that there exists a reparametrization 
$$
\gamma(t) = tv + \theta(t) v^\perp
$$
of an $\wt F$-geodesic. It follows then that
$$
\wt R_0(x,w) = \left( \gamma(1) , \frac{\dot \gamma(1)}{\wt F(\gamma(1), \dot \gamma(1))}
\right)
$$
We define the scaled sections $\wt W_i \subset T\R^2$ for $i \in \Z$ via
$$
\wt W_0 = \{(bv^\perp, v + dv^\perp) \in T\R^2 : b,d \in \R\} \quad \text{and} \quad \wt W_i
= iv + \wt W_0 
$$
together with a map $\wt g : T\R^2 - 0 \to T\R^2 - 0$ with
$$
\wt g(x,w) = \left( x, \frac{w}{\wt F(x,w)} \right)
$$
Observe that the restrictions $\wt g|_{\wt W_i} : \wt W_i \to \wt V_i$
are diffeomorphisms (check that $(x,w) \mapsto (x, \tfrac{\Vert v \Vert^2}{\langle v, w \rangle}w)$ is the inverse
of $\wt g|_{\wt W_i}$ and both are differentiable) and that 
$$
 \left( \wt g|_{\wt W_1} \right)^{-1} \circ \wt R_0 \circ \wt g|_{\wt W_0 } = \wt P_0
$$
for the map $\wt P_0: \wt W_0 \to \wt W_1$ with 
$$
\wt P_0(\gamma(0), \dot \gamma(0)) = (\gamma(1), \dot \gamma(1))
$$
Define diffeomorphisms $\wt l_i : \wt W_i \to \R^2$ via 
$$
\wt l_i (x,w) = \left( \frac{\langle x, v^\perp \rangle}{\Vert v \Vert^2}, 
\frac{\langle w, v^\perp \rangle}{\Vert v \Vert^2}\right)
$$
Then we have 
$$
\wt l_1 \circ \wt P_0 \circ \left( \wt l_0 \right)^{-1} = \varphi_{\wt L}^{0,1}
$$
The conjugacy of the return map $R$ and the time-1 map $\varphi_L^{0,1}$ follows
from the $\Z^2$-invariance of the geodesic flow of $\wt F$.
\end{proof}

\begin{exmp}
Let $\wt F$ be the flat metric, i.e.
$$
\wt F(x,w) = \Vert w \Vert = \sqrt{\langle w, w \rangle}
$$
and $v \in \Z^2 \backslash \{0\}$ prime. The Lagrangian $\wt L$ (see definition \ref{lagrangiandefn})
is then given by 
\begin{flalign*}
\wt L (t,x,r) &= \wt F(tv + xv^\perp, v + rv^\perp) \\
&= \Vert v \Vert \sqrt{1 + r^2}
\end{flalign*}
We obtain that a function $\theta: \R \to \R$ is an Euler-Lagrange solution if and only if
$$
\underbrace{\partial_x \wt L(t,\theta(t), \theta'(t))}_{=0} - \partial_t (\partial_r \wt L(t,\theta(t), \theta'(t))) 
= 0 
$$
which is equivalent to 
$$
\Vert v \Vert \frac{\theta''}{\sqrt{1 + \theta'^2}^3} = 0
$$
and thus $\theta$ is an Euler-Lagrange solution if and only if $\theta''(t) = 0$ for every $t \in \R$.
Consequently the solutions are of the form 
$$
\theta(t) = at + b \quad \text{for } a,b \in \R
$$

Thus, we have for the time-1 map $\varphi_{\wt L}^{0,1}: \R^2 \to \R^2$ of the Lagrangian $\wt L$
$$
\varphi_{\wt L}^{0,1}(x,y) = (x+y,y)
$$

To determine the map $\wt R_0$ let $(x,w) \in \wt V_0$, i.e. 
$(x,w) = (bv^\perp, rv + hv^\perp)$ for $b,h \in \R$ and $r > 0$, such that $\Vert rv + hv^\perp
\Vert = 1$. The Euler-Lagrange solution $\theta$ with $\theta(0) = b$ and $\theta'(0) = \tfrac{h}{r}$
is given by
$$
\theta(t) = \frac{h}{r}t + b
$$
This yields the following reparametrization $\gamma$ of an $\wt F$-geodesic
$$
\gamma(t) = tv + \left( \frac{h}{r}t + b\right) v^\perp
$$
We calculate 
\begin{flalign*}
\wt R_0(x,w) &= \left( \gamma(1), \frac{\dot \gamma(1)}{\Vert \dot \gamma(1) \Vert} \right) \\
&= \left( v + (\tfrac{h}{r} + b)v^\perp , \frac{v + \tfrac{h}{r}v^\perp}{\Vert v+\tfrac{h}{r}v^\perp\Vert} \right) \\
&= \left( x + \frac{1}{r}w, \frac{w}{\Vert w \Vert}\right) \\
&= \left( x + \frac{\Vert v \Vert^2}{\langle w, v \rangle}w, w\right)
\end{flalign*}

To determine the map $\wt P_0$ recall that 
$$
\wt P_0 = \left( \wt g |_{\wt W_1} \right)^{-1} \circ  \wt R_0 \circ  \wt g|_{\wt W_0}
$$
with
$$
\wt g|_{\wt W_i}(x,w) = \left(x, \tfrac{w}{\Vert w \Vert}\right) \quad \text{and}\quad \left( \wt g|_{\wt W_i}\right)^{-1}(x,w) = \left(x, \tfrac{\Vert v \Vert^2}{\langle w, v\rangle}w\right)
$$

For $(x,w) \in \wt W_0$ we obtain 
\begin{flalign*}
\wt P_0(x,w) &= \left( \wt g |_{\wt W_1} \right)^{-1} \circ  \wt R_0 \left(x, \tfrac{w}{\Vert w \Vert}\right) \\
&= \left( \wt g |_{\wt W_1} \right)^{-1} \left(x + \tfrac{\Vert v \Vert^2}{\langle w,v \rangle}w, \tfrac{w}{\Vert w \Vert}\right)\\
&=\left( x + \frac{\Vert v \Vert^2}{\langle w,v \rangle}w, \frac{\Vert v \Vert^2}{\langle w,v \rangle}w \right) \\
&= (x + w, w)
\end{flalign*}
with the last equation following from $(x,w) \in \wt W_0$.
To conclude the example we calculate further 
\begin{flalign*}
\wt l_1 \circ \wt P_0 \circ (\wt l_0)^{-1} (x,y) &= \wt l_1 \circ \wt P_0(xv^\perp, v + yv^\perp) \\
&= \wt l_1(xv^\perp + v + yv^\perp, v + yv^\perp) \\
&= (x+y,y)
\end{flalign*}
with the last line being equal to the time-1 map of $\wt L$.
\end{exmp}

In this example the time-1 map of $\wt L$ is a twist map (the so-called \emph{shear map}).
We will see later that in general the strict convexity of $L$ guarantees that the time-1 map 
$\phi_L^{0,1}$ is conjugated to a finite composition of positive twist maps. This will allow us to apply the 
theory of monotone twist maps to return maps of the geodesic flow.

%%%%%%%%%%%%%%%%%%%%%%%%%%%%%%%%%%%%%%%%%%%%%%

\subsection{Twist maps of the cylinder and Mather's theorem}

In this section we will briefly recall some of the theory on finite compositions of monotone twist maps of the 
cylinder and Mather's theorem about connecting orbits in Birkhoff regions of instability.

\begin{defn}
Let $f: Z \to Z$ be a map of the cylinder with a lift $F: \R^2 \to \R^2$, such that
$F(x,y) = (X(x,y),Y(x,y))$ is a diffeomorphism with
\begin{enumerate}
\item $F$ is isotopic to the identity
\item Twist condition: The map $(x,y) \mapsto (x, X(x,y))$ is a diffeomorphism for every $y$
\item Exact symplectic: $F^*(\lambda) - \lambda = dS$ for some function $S: Z \to \R$.
\end{enumerate}
Then $f : Z \to Z$ is called a \emph{monotone twist map of the cylinder.} The map $f$ is called \emph{positive} 
if $\partial_y X > 0$ everywhere. Here, $\lambda = ydx$ is the Liouville form on $\R^2$. 
\end{defn}

For further reading on twist maps we refer the reader to \cite{gole}. The following definitions
go along the lines of \cite{mather91}.
The following theorem (originally due to Birkhoff \cite{birkhoff20}) can be found in the Appendix of \cite{mather91} for 
maps $f \in \Pone$. 

\begin{thm} \label{twist-lipschitz}
Let $f \in \mathcal{P}^1$. Any $f$-invariant homotopically non-trivial Jordan curve (homeomorphic to $S^1$) 
$\Gamma$ in the infinite cylinder is the graph of a Lipschitz function $u: S^1 \to \R$. 
\end{thm}

By $\Pone$ we denote the set of finite compositions of positive monotone twist maps of the cylinder.

\begin{defn}
Let $f \in \Pone$. A \emph{Birkhoff region of instability} $B \subset Z$ is a compact $f$-invariant subset, such that
\begin{enumerate}
\item The boundary of $B$ consists of two components $\Gamma_-$ and $\Gamma_+$, each homeomorphic to a circle
and each non-contractible in $Z$
\item If $\Gamma \subset B$ is any $f$-invariant subset, homeomorphic to a circle, non-contractible, then 
$\Gamma = \Gamma_-$ or $\Gamma = \Gamma_+$
\end{enumerate}
\end{defn}

For every $f \in \Pone$ there is a variational principle $h$ (cf \cite{mather91} chapter 1) 
and consequently one can define the notion of a minimal orbit of
$f$. Much like if $f$ is a twist map one can prove that  every minimal orbit has a 
rotation number. This let's us define the set $\M_{f, \omega}$ of minimal orbits of $f$ with
rotation number $\omega$. 
Let $\Gamma \subset Z$ be an $f$-invariant circle (and consequently after theorem \ref{twist-lipschitz} a Lipschitz graph) with rotation number $\omega$. It 
follows from \cite{mather91} proposition 2.8 that $\Gamma \subset \M_{f,\omega}$. If $\omega$ is 
irrational it follows from \cite{mather91} proposition 2.6 that the projection $\pi: Z \to S^1$ restricted to 
$\M_{f,\omega}$ is injective, and hence we have $\Gamma = \M_{f,\omega}$.\\

A special case of theorem 4.1 in \cite{mather91} is the following connecting theorem by J. Mather

\begin{thm} \label{mather1}
Let $f \in \Pone$. If $f$ has a Birkhoff region of instability $B$, such that the rotation numbers 
$\omega_- < \omega_+$ corresponding to the boundary graphs 
$\Gamma_-$ (respectively $\Gamma_+$) are irrational, then there exists an $f$-orbit $\mathcal{O}$
that is $\alpha$-asymptotic to $\Gamma_-$ and $\omega$-asymptotic to $\Gamma_+$.
\end{thm} 

The orbit of a point $x \in S^1 \times \R$ is called \emph{$\omega$-asymptotic}  
to $\Gamma_+$ if 
$d(f^n(x), \Gamma_+)$ converges to zero as $n \to \infty$. It is called \emph{$\alpha$-asymptotic} to 
$\Gamma_-$ if $d(f^{-n}(x), \Gamma_-)$ converges to zero as $n \to \infty$.\\

The next theorem can be found as theorem 39.1 in \cite{gole} in a more general setting. It allows us to determine
when the time-$(s,t)$ maps of a time-periodic Hamiltonian on $Z$ are twist maps. 

\begin{thm} \label{globalleg}
Let $H: S^1 \times S^1 \times \R \to \R$ be a smooth Hamiltonian, such that the maps
$h_{t,x}: \R \to \R$ with 
$$
h_{t,x}: p \mapsto \partial_p H(t,x,p)
$$
are diffeomorphisms for every $(t,x) \in S^1 \times S^1$ (global Legendre condition). Then, given 
any compact set $K \subset S^1 \times \R$ and starting time $a$, there exists an $\epsilon_0 > 0$, 
such that, for all $0 < \epsilon < \epsilon_0$ the time-$(a, a+\epsilon)$ map of $H$ is a 
twist map on $K$.
\end{thm}

\section{Proof of the main result}

To prove the main theorem we will define a time-periodic Lagrangian $L$ on $S^1$, 
such that the Euler-Lagrange solutions of $L$ correspond to reparametrizations of $F$-geodesics, 
whose lifts are graphs over the 
euclidean line $e_1 \cdot \R  \subset \R^2$. After a perturbation of the Lagrangian outside a tube
$S^1 \times (-R,R)$ we will be able to work with the conjugated Hamiltonian $H$. 
Theorem \ref{globalleg} will guarantee that the time-1 map $\psi$ of $H$ is in $\Pone$, i.e. 
a composition of finitely many positive monotone twist maps.
We will then see that the instability region of the geodesic flow corresponds to a Birkhoff region 
of instability of $\psi$ and hence Mather's connecting theorem \ref{mather1} will guarantee the 
existence of an orbit connecting the boundary graphs of the Hamiltonian (respectively Lagrangian) system.
Using the correspondence of 
Euler-Lagrange solutions and reparametrizations of geodesics we obtain a connecting geodesic
with the required asymptotic behaviour.\\

Let $v \in \Z \backslash \{ 0 \}$. Like in the previous section we associate
a smooth and 
strictly convex time-periodic Lagrangian $L$ on $S^1$ to the Finsler metric $F$ via 

$$
L(t,x,r) = F\left( tv + xv^\perp, v + rv^\perp\right)
$$

We perturb the Lagrangian for large values of $|r|$:
For $R > 0$ we can find a constant $D > 0$, such that there is an extension $L_R$ of $L|_{S^1 \times S^1 \times (-R,R)}$
to $S^1 \times S^1 \times \R$ with 
\begin{enumerate}
\item $L_R(t,x,r) = \frac{D}{2} r^2$ for large values of $|r|$
\item $L_R$ is strictly convex with second partial derivative 
bounded and bounded away from zero, i.e. 
$0 < \frac{1}{C} < \partial_{rr} L_R < C$ for a constant $C > 0$
\end{enumerate}

\begin{rmk}
Observe that the Euler-Lagrange solutions of $L_R$ for large values of $|r|$ are just straight lines, and hence the integral
curves of the Euler-Lagrange vector field exist for all times. To see this lift $L_R$ to a 1-periodic 
Lagrangian on $\R$ and observe that for large values of $|r|$ the Euler-Lagrange equation 
reduces to $\ddot \gamma = 0$ for functions $\gamma: \R \to \R$. Thus, for large $|r|$ the solutions
are straight lines. Consequently, for
large $C >0$ we have that the compact sets $S^1 \times S^1 \times [-C,C]$ are invariant under the
Euler-Lagrange flow, and hence every solution exists for all times, i.e. the flow is complete.
\end{rmk}

The following theorem is a corollary of theorem \ref{schrthm}.

\begin{thm} \label{schroeder1}
Let $F$ be a Finsler metric on $T^2$ and let $L_R$ be the above Lagrangian for an $R > 0$.
Let $\theta: \R \to \R$ be a smooth function with $-R < \theta' < R$ and let $\gamma: \R \to \R^2$ be the curve given by 
$$
\gamma(t) = t v + \theta(t) v^\perp
$$
Then $\gamma$ is a reparametrization of a lift of an $\wt F$-geodesic if and only if $\theta$ is an Euler-Lagrange solution of
$\wt L_R$.
\end{thm}

\begin{proof}
The theorem follows immediately for a curve $\theta: \R \to \R$ with $-R < \theta' < R$ since
the perturbed Lagrangian $L_R$ agrees with the unperturbed Lagrangian $L$ in a 
neighbourhood of the curve. Consequently, the Euler-Lagrange equations agree on this 
neighbourhood and hence $\theta$ is an Euler-Lagrange solution for $L_R$ if and only if 
$\theta$ is an Euler-Lagrange solution for $L$. 

\end{proof}

To define the associated Hamiltonian $H_R: S^1 \times S^1 \times \R \to \R$ of the Lagrangian $L_R$ we need to 
consider the Legendre-transform $\Leg_t : S^1 \times \R \to 
S^1 \times \R$ with
$$
\Leg_t : (x,r) \mapsto (x, \partial_r L(t,x,r))
$$

For a moment we will omit the dependence on $R$ and just write $L$ instead of $L_R$. 
Note that the Legendre-transform is a diffeomorphism because $\partial_{rr}L$ is bounded away from zero. 
For the above Lagrangian $L$ we can define the associated Hamiltonian $H : S^1 \times S^1 \times \R \to \R$
via 
\begin{equation} \label{eq:hamdefn}
H(t,x, \partial_rL(t,x,r)) = \partial_r L(t,x,r) r - L(t,x,r)
\end{equation}
The time-$(s,t)$ maps of the Lagrangian and associated Hamiltonian for $s < t$ are conjugated via the 
Legendre transform 
$$
\psi^{s,t} = \Leg_t \circ \varphi^{s,t} \circ \Leg_s^{-1}
$$

It follows from \eqref{eq:hamdefn} and the strict convexity of $L$ that $\frac{1}{C} < \partial_{pp} H < C$ 
for a constant $C >0$.

\begin{prop}
The time-1 map $\psi:Z \to Z$ of the associated Hamiltonian $H_R: Z \to \R$ of the Lagrangian
$L_R$ is in $\Pone$.
\end{prop}

\begin{proof}
Since $\partial_{pp} H(t,x,p) \neq 0$
the maps $h_{t,x}$ are local diffeomorphisms. Injectivity follows from $\partial_{pp} H >0$ and 
surjectivity follows from $\partial_{pp} H$ being bounded away from zero.
Consequently, for every large tube $S^1 \times [-P,P]$ (those tubes are actually invariant under the
time-dependent flow of $H$ if we chose $P$ to be large enough) there is an $\epsilon >0$, such that the
restriction
of $\psi_H^{0,\epsilon}$ to the tube is a twist map. Outside those tubes the Lagrangian $L$ reduces 
to $L(t,x,r) = \tfrac{D}{2}r^2$ and hence with the Legendre-transform being equal to
$ \Leg_t(x,r) = (x, Dr)$ we find that for the Hamiltonian $H$ it holds that
$$
H(t,x,y) = \frac{y^2}{2D} \quad \text{for large } |y|
$$
Thus, the Hamiltonian vector field is identical to $X_H(t,x,y) = (\tfrac{y}{D},0)$ outside large tubes
and hence the time-1 map outside large tubes is identical to a shear-map of twist 
$\tfrac{1}{D}$.
Since the starting time $a$ in theorem \ref{globalleg} is arbitrary we can piece the time-1 map
$\psi = \psi_H^{0,1}: S^1 \times \R \to S^1 \times \R$ together from finitely many twist maps,
i.e. there exist real numbers $0 = \epsilon_0 < \dots < \epsilon_n = 1$, such that
$$
\psi = \psi_H^{\epsilon_{n-1}, \epsilon_n }\circ \dots \circ \psi_H^{\epsilon_0, \epsilon_1}.
$$
and $\psi_H^{\epsilon_i, \epsilon_{i+1}}$ is a twist map for every $i$.
\end{proof}

\begin{lem} \label{torustocircle}
If $\Lambda \subset ST^2$ is an invariant torus with bounded direction (with respect to $v$) and
$R>0$ is chosen large enough, then the time-1 map of the Lagrangian $L_R$ 
has a corresponding invariant circle $\Gamma$. 
If there are no closed geodesics on $\Lambda$ the invariant circle has irrational rotation
number. 
\end{lem}

\begin{rmk}
By "Corresponding" in lemma \ref{torustocircle} we mean that every Euler-Lagrange solution 
$t \mapsto (\theta_s(t), \partial_t \theta_s(t))$ of $\wt L_R$ with 
$$
(\theta_s(0), \partial_t\theta_s(0)) = (s, \partial_t \theta_s(0)) \in \Gamma
$$
corresponds to a reparametrization $\gamma_s$ of the geodesic in $\Lambda$ with $\gamma_s(0) = sv^\perp$.
\end{rmk}

\begin{proof} \emph{(of lemma \ref{torustocircle})}
It is useful to see $\Lambda$ as a $\Z^2$-periodic continuous graph in 
$S\R^2 \cong \R^2 \times S^1$,
which is invariant under the geodesic flow $\phi : S\R^2 \times \R \to S\R^2$ of the lifted Finsler
metric. 
First we want to define a reparametrization of the restricted geodesic flow 
$\phi: \Lambda \times \R \to \Lambda$ on $\Lambda$.
Since the projections of the geodesics on $\Lambda$ are graphs over $v\R$ and because of the periodicity of $\Lambda$,
we have 
\begin{equation} \label{eq:graph1}
\langle \pi_2 \circ \phi(w,t) , v \rangle \geq const > 0
\end{equation}
for every $(w,t) \in \Lambda \times \R$. 
We define a smooth map $k: S\R^2 \times \R \to S\R^2 \times \R$ via 
$$
k(w,t) = \left(w, \int_0^t \langle \pi_2 \circ \phi(w,s), v \rangle ds \right)
$$

It follows from \eqref{eq:graph1} that $k$ is a local diffeomorphism in a neighbourhood of every $(w,t) \in \Lambda 
\times \R$. Furthermore, it follows also from \eqref{eq:graph1} that the restriction of $k$ to $\Lambda \times \R$ is 
bijective and consequently the map $k: \Lambda \times \R \to \Lambda \times \R$ is a homeomorphism. Hence, 
we have the inverse $k^{-1} : \Lambda \times \R \to \Lambda \times \R$, which is continuous and differentiable 
in the second component. 
We define a reparametrization $\hat \phi : \Lambda \times \R \to \Lambda$ of the geodesic flow via 
$$
\hat \phi (w,t) = \phi \circ k^{-1} (w, \Vert v \Vert^2 t)
$$

Note, that since the graph $\Lambda$ is $\Z^2$-periodic, we have for $z \in \Z^2$
\begin{equation} \label{eq:graphinvariance}
\pi_1 \circ \hat \phi(w',t) = \pi_1 \circ \hat \phi(w,t) + z \quad \text{if } \pi_1(w') = \pi_1(w) + z
\end{equation}

From $k \circ k^{-1} = id$ we calculate that 
\begin{equation} \label{eq:graphkuk}
\langle \partial_t (\pi_1 \circ \hat \phi)(w,t) , v \rangle = \Vert v \Vert^2
\end{equation}
for every $(w,t) \in \Lambda$. 
Hence we can write the map $\pi_1 \circ \hat \phi : \Lambda \times \R \to \R^2$ as
\begin{equation} \label{eq:graphdarstellung}
\pi_1 \circ \hat \phi(w,t) = \left(\tfrac{1}{\Vert v \Vert^2}\langle \pi_1(w) , v \rangle + t\right) v + \theta(w,t) v^\perp
\end{equation}
where $\theta: \Lambda \times \R \to \R$ is continuous and differentiable in the second component. 
For $i \in \Z$ we define subsets $\Lambda_i \subset \Lambda$ via 
$$
\Lambda_i = \{w \in \Lambda \ | \ \pi_1(w) \in \R v^\perp + iv\}
$$
Note, that for $w \in \Lambda_0$ we get from \eqref{eq:graphdarstellung}
\begin{equation} \label{eq:graphlambda0}
\pi_1 \circ \hat \phi(w,t) = tv + \theta(w,t) v^\perp
\end{equation}

As an intermediate step we will show that 
\begin{equation} \label{eq:graphadditive}
\hat \phi(w,t + t') = \hat \phi(\hat \phi(w,t'), t)\quad  \text{ for every } t, t' \in \R
\end{equation}

To see this we define curves $\gamma_1, \gamma_2: \R \to \Lambda$ via 
$$
\gamma_1(t) = \hat \phi(w, t' + t)
$$ 
and 
$$
\gamma_2(t) = \hat \phi(\hat \phi(w,t'), t)
$$
Since the curves are both reparametrizations of a geodesic on $\Lambda$ with $\gamma_1(0) = \gamma_2(0)$ they are 
reparametrizations of the same geodesic. From equation \eqref{eq:graphkuk} it follows for every $t$ that 
$$
\langle \partial_t(\pi_1 \circ \gamma_1)(t), v \rangle = \langle \partial_t(\pi_1 \circ \gamma_2)(t), v \rangle
$$
together with $\gamma_1(0) = \gamma_2(0)$ we obtain
$$
\langle \pi_1 \circ \gamma_1(t), v \rangle = \langle \pi_1 \circ \gamma_2(t) , v \rangle 
$$
for every $t \in \R$. 
Since every projection of a geodesic in $\Lambda$ is a graph over $\R v$ this implies that
$$
\pi_1 \circ \gamma_1(t) = \pi_1 \circ \gamma_2(t)
$$
for every $t \in \R$. Finally, since $\Lambda$ is a graph over $\R^2$ we have 
$\gamma_1 = \gamma_2$ and thus we have proven \eqref{eq:graphadditive}.
Next we will prove that
\begin{equation} \label{eq:graphlambdai}
\pi_1 \circ \hat \phi(w,1) \in \Lambda_{i+1} \text{ for every } w \in \Lambda_i 
\end{equation}

To see this note that from \eqref{eq:graphdarstellung} and $\langle \pi_1 (w) , v\rangle = i \Vert v \Vert^2$ it follows that
$$
\langle \pi_1 \circ \hat \phi(w,1) , v \rangle = (i+1) \Vert v \Vert^2
$$

We can now define a graph $\Gamma \subset S^1 \times \R$, which we will later prove to be an invariant circle of the
time-1 map of $L_R$. We set 
$$
\Gamma = \{ (\theta(w,0), \partial_t\theta (w,0)) \ | \ w \in \Lambda_0\}
$$
To see that $\Gamma$ is actually a graph in $S^1 \times \R$ note that for every $w \in \Lambda_0$ we have
$\pi_1(w) = sv^\perp$ for an $s \in \R$. Consequently it follows from \eqref{eq:graphlambda0} that 
$\theta(w,0) = s$ if $\pi_1(w) = sv^\perp$. To see that $\Gamma$ is periodic (i.e. a subset in  $S^1 \times \R$) note that
if $w, w' \in \Lambda_0$ with $\pi_1(w) = sv^\perp$ and $\pi_1(w') = (s+1)v^\perp$ we have 
$\pi_1 \circ \hat \phi(w',t) = \pi_1 \circ \hat \phi(w,t) + v^\perp$ because of \eqref{eq:graphinvariance}. Thus it 
follows from \eqref{eq:graphlambda0} that $\partial_t \theta(w,0) = \partial_t \theta(w',0)$ and hence $\Gamma$ can be
seen as the graph of the 1-periodic function $h: S^1 \to \R$ with 
$$
h(s) = \partial_t \theta(\Lambda(sv^\perp), 0)
$$
Note, that $\Gamma$ is a continuous graph since $h$ is a composition of continuous functions. We can thus 
pick an $R > 0$ large enough, so that $\Gamma \subset S^1 \times (-R,R)$.
To see now that $\Gamma$ is invariant under the time-1 map $\varphi: S^1 \times \R \to S^1 \times \R$ of the 
Lagrangian $L_R$ observe that for a $w \in \Lambda_0$ it follows from \eqref{eq:graphdarstellung} and 
\eqref{eq:graphadditive} that 

\begin{flalign*}
\theta(w, 1 + t) &= \frac{1}{\Vert v \Vert^2} \langle \pi_1 \circ \hat \phi(w, 1+t) , v^\perp \rangle \\
&= \frac{1}{\Vert v \Vert^2} \langle \pi_1 \circ \hat \phi(\hat \phi(w,1), t) , v^\perp \rangle
\end{flalign*}

Since $\hat \phi(w,1) \in \Lambda_1$ (see \eqref{eq:graphlambdai}) and because of the invariance
\eqref{eq:graphinvariance} there exists a $w' \in \Lambda_0$, such that the last line is equal to

\begin{flalign*}
&= \frac{1}{\Vert v \Vert^2} \langle \pi_1 \circ \hat \phi(w', t) + v , v^\perp \rangle \\
&= \frac{1}{\Vert v \Vert^2} \langle \pi_1 \circ \hat \phi(w', t) , v^\perp \rangle \\
&= \theta(w',t)
\end{flalign*}
for every $t \in \R$. This implies that $\theta(w,1) = \theta(w',0)$ and 
$\partial_t\theta(w,1) = \partial_t\theta(w',0)$ and hence using theorem \ref{schroeder1} the time-1 
map $\varphi$ with 
$$
\varphi(\theta(w,0), \partial_t\theta(w,0)) = (\theta(w,1), \partial_t \theta(w,1))
$$ 
maps $\Gamma$ onto itself.
Applying the Legendre transform to the graph $\Gamma$ yields an invariant graph of the time-1 map $\psi$ of 
the associated Hamiltonian $H_R$. 
Assume now, that $\Gamma$ has rational rotation number. Consequently there exists a periodic 
point $x = (\theta(w,0), \partial_t \theta(w,0))$ in $\Gamma$. This implies that the geodesic with 
initial velocity $w$ is closed, which is a contradiction to $\Lambda$ having no closed geodesics.
\end{proof}

\begin{lem} \label{circletotorus}
Let $U \subset ST^2$ be an instability region with boundary graphs $\Lambda_-$ and $\Lambda_+$. Assume that $U$ has 
bounded direction (with respect to $ v = e_1$). If $R >0$ is chosen large enough, such that the time-1 
map of the Lagrangian $L_R$ has two corresponding invariant
circles $\Lambda_-$ and $\Lambda_+$, then the only invariant circles in the subset $B \subset S^1 \times \R$ 
enclosed by $\Lambda_-$ and $\Lambda_+$ are the boundary circles $\Lambda_-$ and $\Lambda_+$.
\end{lem}

\begin{proof}

To prove this theorem we have to show that if we have an invariant circle $\Gamma$ of the map 
$\varphi$, which is not equal to $\Gamma_-$ or $\Gamma_+$ then we also have an
$\phi$-invariant torus $\Lambda$ not equal to $\Lambda_-$ or $\Lambda_+$.
Now, assume that $\Gamma \subset S^1 \times \R$ is such an invariant circle of $\varphi$.  It follows from
theorem \ref{twist-lipschitz} that $\Gamma$ is equal to the graph $\{(s, h(s)) | s \in S^1\}$  of a Lipschitz function
$h: S^1 \to \R$. For every fixed $s \in S^1$ the map $t \mapsto \varphi^{0,t}(s,h(s))$ is a trajectory of the 
Euler-Lagrange flow on $S^1 \times \R$. We define maps  $\theta_s: \R \to \R$ via
$$
(\theta_s(t), \theta_s'(t)) := \varphi^{0,t}(s,h(s))
$$
Note that the maps $(s,t) \mapsto \theta(s,t)$ and $(s,t) \mapsto \theta_s'(t)$ are continuous, since $\theta_s$ and
$\theta_s'$ are the compositions of $\varphi^{0,t}(s,h(s))$ with the projections of $S^1 \times \R$ onto $S^1$ respectively
$\R$. 
We extend $\theta$ to a map $\theta: \R^2 \to \R$, $(s,t) \mapsto \theta_s(t)$ by seeing $\Gamma$ as a 
1-periodic graph in $\R^2$ and using the Euler-Lagrange flow $\varphi$ of the lifted Lagrangian $\wt L$ in the 
definition of $\theta$. We then have 
$$
(\theta_{s+1}(t), \theta_{s+1}'(t)) = (\theta_s(t) + 1, \theta_s'(t))
$$
Observe that since $\varphi^{0,0} = id$ we have $\theta_s(0) = s$ for every $s \in \R$. 
Note that for $r \neq p$ we have $\theta_r(t) \neq \theta_p(t)$. To see this assume that we have 
$\theta_r(t) = \theta_p(t)$. Since the set $\varphi^{0,t}(\Gamma) = \{(\theta_s(t), \theta_s'(t)) | s \in \R \}$ is invariant
under the map $\varphi^{t,t+1}$ it is a Lipschitz graph. Hence it also holds that $\theta_r'(t) = \theta_p'(t)$.
We have the associated reparametrizations of geodesics $\gamma_r(t) = tv + \theta_r(t)v^\perp$ and 
$\gamma_p(t) = tv + \theta_p(t) v^\perp$. This would imply that the geodesics meet at the point 
$\gamma_r(t) = \gamma_p(t)$ with $\dot \gamma_r(t) = \dot \gamma_p(t)$, which is not possible because geodesics
can not become tangent to each other. 
As mentioned above, since the sets $\Gamma_t := \varphi^{0,t}(\Gamma)$ are Lipschitz graphs there exist periodic
Lipschitz functions $h_t : \R \to \R$ with $\Gamma_t = \{(s,h_t(s) | s \in \R\}$ with 
$$
h_t(\theta_s(t)) = \theta_s'(t).
$$
We define a function $g: \R^2 \to \R$ via
$$
g: (s,t) \mapsto h_t(s).
$$
We will now show that $g$ is continuous. To see this take real sequences $t_n \to t$ and $s_n \to s$. Since $\Gamma_t$
is a Lipschitz graph there exists exactly one $r_n \in \R$ for every $n \in \N$ with 
$$
\theta_{r_n}(t_n) = s_n.
$$ 
The sequence $r_n$ is bounded. To see this observe that it follows from the continuity of the map 
$(s,t) \mapsto \theta_s(t)$ and the graph property of the sets $\Gamma_t$ that the map 
$s \mapsto \theta_s(t)$ is strictly increasing for every fixed $t$. Now, assume that the sequence $r_n$ is 
unbounded. Without loss of generality we assume that $r_n$ is unbounded from above with $r_n \geq n$ for 
every $n \in \N$. Then we have 
$$
\theta_{r_n}(t_n) \geq \theta_n(t_n) = \theta_0(t_n) + n 
$$
and thus the sequence $\theta_{r_n}(t_n)$ is also unbounded. This contradicts the fact that $s_n$ converges 
to $s$ and consequently the sequence $r_n$ has to be bounded. 
To see that $r_n$ 
actually converges assume that there are two limit points $r = \lim_{i \to \infty} r_{n_i}$ and $p = \lim_{j \to \infty}
r_{n_j}$. Since we have $\theta_{r_{n_i}}(t_{n_i}) = s_{n_i}$ it follows from the continuity of $\theta$ that 
$\theta_r(t) = s$. Analogously we obtain $\theta_p(t) = s$. As above it follows that $r = p$.
To see that $g$ is continuous observe that
$$
h_{t_n}(s_n) = h_{t_n}(\theta_{r_n}(t_n)) = \theta_{r_n}'(t_n)
$$
and hence it follows from the continuity of $(s,t) \mapsto \theta_s'(t)$ that 
$$
\lim_{n \to \infty} h_{t_n}(s_n) = \lim_{n \to \infty} \theta_{r_n}'(t_n) = \theta_r'(t) = h_t(s) 
$$
Thus we have proven that $g: (s,t) \mapsto h_t(s)$ is continuous.
To construct an invariant graph in $ST^2$, observe that since the reparametrized geodesics $\gamma_s$ with
$\gamma_s(t) = tv + \theta_s(t)v^\perp$ can not 
cross and since $\gamma_{s+1} = \gamma_s + v^\perp$ the geodesics $\gamma_s$ foliate $\R^2$. This allows
us to define a graph $\Lambda \subset S\R^2$ via 
$$
\Lambda = \left\{ \left( \gamma_s(t) , \frac{\dot \gamma_s(t)}{F(\dot \gamma_s(t))}\right)  : (s,t) \in \R^2 \right\}
$$
Furthermore, 
$\Lambda$ is $\phi$-invariant since it only consists of images of geodesics. 
We define a map $X: \R^2 \to \R^2$  via 
$$
X: x \mapsto v + h_{\left\langle x, \frac{v}{\Vert v \Vert^2} \right\rangle} \left( \left\langle x, \frac{v^\perp}{\Vert v \Vert^2} \right\rangle \right) v^\perp
$$
Observe that $X$ is continuous (because $h$ is)  and maps $\gamma_s(t)$ to $\dot \gamma_s(t)$. Hence $\Lambda$ is 
the graph of the continuous function 
$$
x \mapsto \frac{X(x)}{F(X(x))}
$$
What is left to prove is that $X$ is $\Z^2$ periodic, i.e. $\Lambda$ is actually a graph in $ST^2$. We will show that if we 
assume our instability region to have bounded direction with respect to $v \in \Z^2 - \{0\}$ then $X$ is 
$v\Z \times v^\perp \Z$-periodic. As a special case we get that $X$ is $\Z^2$-periodic if $v$ is chosen to be
equal to $e_1$. 
To prove that $X$ (and thus $\frac{X}{F(X)}$) is $v\Z \times v^\perp \Z$-periodic observe that
for every $(s,t) \in \R^2$ we have 
\begin{flalign*}
\gamma_s(t) + v^\perp &= tv + \theta_s(t)v^\perp + v^\perp \\
&= tv + (\theta_s(t) + 1) v^\perp \\
&= tv + \theta_{s+1}(t) v^\perp \\
&= \gamma_{s+1}(t)
\end{flalign*}
Since $\theta_s'(t) = \theta_{s+1}'(t)$ we have  $\dot \gamma_s(t) = \dot \gamma_{s+1}(t)$ and thus $X$ is $v\Z$-periodic.
From the invariance of $\Gamma_t$ under $\varphi^{t,t+1}$ we obtain that for every $s \in \R$ there exists an $r \in \R$ 
with 
$$
\varphi^{t,t+1} (\theta_r(t), \theta_r'(t)) = (\theta_s(t), \theta_s'(t))
$$
Thus we have 
\begin{flalign*}
\gamma_s(t) + v &= (t+1)v + \theta_s(t)v^\perp \\
&=(t+1) v + \theta_r(t+1) v^\perp \\
&= \gamma_r(t+1)
\end{flalign*}

And since $\theta_r'(t+1) = \theta_s'(t)$ we have $\dot \gamma_r(t+1) = \dot \gamma_s(t)$. Hence we have shown the
$v\Z$-periodicity of $X$.

\end{proof}

We can now prove the main theorem.

\begin{proof} (theorem \ref{main})

If $R >0$ is large enough lemma \ref{torustocircle} guarantees that the time-1 map $\varphi$ of $L_R$ has 
two invariant circles $\Gamma_-$ and $\Gamma_+$ in $S^1 \times (-R,R)$, which correspond to the invariant boundary 
graphs. With lemma \ref{circletotorus} it follows that the region bounded by the corresponding invariant circles of the 
Hamiltonian time-1 map is a Birkhoff region of instability. Thus, after going back to the Lagrangian setting, we 
obtain from theorem \ref{mather1} a point $x$ in the area between the boundary circles, and sequences $t'_n \to \infty$, $s'_n \to 
-\infty$ with 
$$
\lim_{n \to \infty} \varphi^{0,t'_n}(x) \in \Gamma_+
$$
and
$$
\lim_{n \to \infty} \varphi^{0,s'_n}(x) \in \Gamma_-
$$
Using theorem \ref{schroeder1} we obtain after reparametrization a geodesic $c: \R \to T^2$ and sequences $t_n \to \infty$,
$s_n \to -\infty$ with 
$$
\lim_{n \to \infty} \dot c(t_n) \in \Lambda_+
$$
and 
$$
\lim_{n \to \infty} \dot c(s_n) \in \Lambda_-
$$
\end{proof}

%%%%%% %%%%%% %%%%%% BIBLIOGRAPHY %%%%%% %%%%%% %%%%%% 

\end{document}